\documentclass[11pt,reqno,tbtags]{amsart}
\usepackage{graphicx}
\usepackage{amsmath,amsthm}
\usepackage{amstext}
\usepackage{alltt} 
\usepackage{array,longtable}
\usepackage[T1]{fontenc}
\usepackage[english]{babel}
\usepackage{url}

\theoremstyle{plain}
\newtheorem{theorem}{Theorem}[section]
\newtheorem{lemma}[theorem]{Lemma}

\newtheorem{corollary}[theorem]{Corollary}
\newtheorem{conjecture}[theorem]{Conjecture}
\newtheorem{observation}[theorem]{Observation}

\theoremstyle{definition}

\theoremstyle{remark}

\newtheorem{problem}{Problem}

\begin{document}
\title{On snarks that are far from being 3-edge colorable}

\date{\today} 

\author{Jonas Hägglund}
\address{Department of Mathematics and Mathematical Statistics, Ume\aa\ University, SE-901 87 Ume\aa, Sweden}
\email{jonas.hagglund@math.umu.se}

\begin{abstract} 
	In this note we construct two infinite snark families which have high oddness and low circumference compared to the number of vertices. 
	Using this construction, we also give a counterexample to a suggested strengthening of Fulkerson's conjecture by showing that the Petersen graph is not the only cyclically 4-edge connected cubic graph which require at least five perfect matchings to cover its edges. Furthermore the counterexample presented has the interesting property that no 2-factor can be part of a cycle double cover. 
\end{abstract} 

\maketitle


\section{Introduction}
	A cubic graph is said to be \emph{colorable} if it has a 3-edge coloring and \emph{uncolorable} otherwise. 
	A \emph{snark} is an uncolorable cubic cyclically 4-edge connected graph. It it well known that an edge minimal counterexample (if such exists) to some classical conjectures in graph theory, such as the cycle double cover conjecture  \cite{MR538060, MR0325438}, Tutte's 5-flow conjecture \cite{Tutte:1954bh} and Fulkerson's conjecture \cite{MR0294149}, must reside in this family of graphs. 
	
	There are various ways of measuring how far a snark is from being colorable. One such measure which was introduced by Huck and Kochol \cite{Huck:1995fk} is the \emph{oddness}. The oddness of a bridgeless cubic graph $G$ is defined as the minimum number of odd components in any 2-factor in $G$ and is denoted by $o(G)$. Another measure is the \emph{resistance} of $G$, $r_3(G)$, which was introduced by Steffen \cite{Steffen:2004uq} and is defined as the minimal number of edges that needs to be removed from $G$ in order to obtain a 3-edge colorable graph, i.e. $r_3(G):=\min\{|M| :M\subset E(G)\textrm{ and }\chi'(G-M) = 3\}$. It is easy to see that $r_3(G)\leq o(G)$. It is also known that there exists families of snarks where these measures are arbitrary large and arbitrary far apart  \cite{Steffen:2004uq}. 
	
	Although snarks can be arbitrary far from being colorable in the sense of oddness and resistance, this might not be the case when we consider other uncolorablilty measures. The \emph{perfect matching index} $\tau(G)$ of a cubic graph $G$ was introduced in \cite{Fouquet:2009fk} and is defined as the minimum number of perfect matchings needed to cover $E(G)$. A famous conjecture by Fulkerson asserts that every bridgeless cubic graph has a double cover by six perfect matchings \cite{MR0294149}, and if true, would imply that $\tau(G)\leq 5$ for every bridgeless cubic graph $G$. Recently Mazzuoccolo showed that these two statements are in fact equivalent \cite{pre05957651}. However it is not known if there exists a constant $k$ such that $\tau(G)<k$ for all snarks $G$. 
	
	It is easy to see that $\tau(P)=5$, where $P$ is the Petersen graph. It has been proposed \cite{Mazzuoccolo:2011fk, Fouquet:2009fk} that a possible strengthening of Fulkerson's conjecture could be the assertion that the Petersen graph is in fact the only snark with this property and that all other snarks have $\tau=4$. 
	\begin{conjecture}[ \cite{Mazzuoccolo:2011fk, Fouquet:2009fk}] \label{conj:false}
		Let $G$ be a cyclically 4-edge connected cubic graph. If $\tau(G)=5$, then $G$ is the Petersen graph.  
	\end{conjecture}
	In this paper we present a counterexample to this conjecture. Furthermore we note that this counterexample has the interesting property that no 2-factor can be part of a cycle double cover. We also give simple constructions for two infinite families of snarks with high oddness and resistance compared to the number of vertices. 
		
\section{The construction} \label{section:construct}
	The following simple lemma is well known (see e.g. \cite{MR0382052,Nedela:1996kx, Steffen:2004uq}) and very useful when studying edge colorability of cubic graphs. 
	
	\begin{lemma}[Parity lemma]\label{plemma}
		Let $\phi:E(G)\rightarrow \{1,2,3\}$ be a 3-edge-coloring of a cubic graph $G$. Then for every edge cut $M$ in $G$ we have that $|\phi^{-1}(1)\cap M|\equiv |\phi^{-1}(2)\cap M| \equiv|\phi^{-1}(3)\cap M|\equiv |M|(\mod 2)$
	\end{lemma}
	
	Following the notation from \cite{MR1385382} we say that a \emph{semiedge} is an edge which is incident to exactly one vertex or one vertex and another semiedge. In the latter case we simply identify it with a normal edge. A \emph{multipole} $M$ is a triple $M=(V, E,S)$ where $V=V(M)$ is the vertex set, $E=E(M)$ is the edge set and $S=S(M)$ is the set of semiedges. A multipole with $k$ semiedges is called a $k$-pole.

	\begin{figure}[h!t]	
		\includegraphics[scale=0.4]{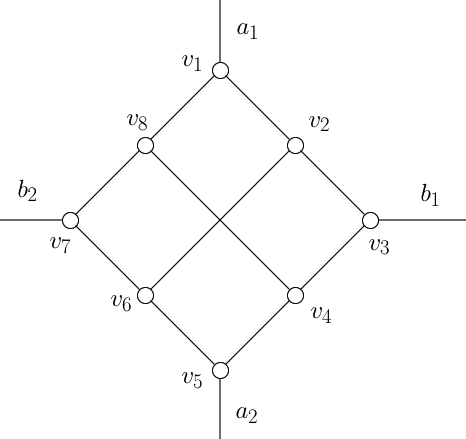}
		\caption{The 4-pole $B$ is constructed from the Petersen graph by removing two adjacent vertices.}
		\label{B-fig}
	\end{figure}

	Let $B$ be the 4-pole obtained by removing two adjacent vertices from the Petersen graph (see Figure \ref{B-fig}). 
	
	\begin{lemma} \label{samecol}
		In every 3-edge coloring of the 4-pole $B$ the semiedges $a_1$ and $a_2$ receive the same color. 
	\end{lemma}
	\begin{proof}
		Assume that $a_1$ and $a_2$ have different colors in a 3-edge coloring $\phi$ of $B$. W.l.o.g. we may assume that $\phi(a_1) = 1$ and $\phi(a_2) = 2$. Now, by Lemma \ref{plemma}, one of $b_1$ and $b_2$ must have color 1 and the other color 2. W.l.o.g. we can assume that $\phi(b_1) = 1$ and $\phi(b_2) = 2$. Furthermore we can assume w.l.o.g. that $\phi(v_1v_2) = 2$ and $\phi(v_2v_3) = 3$. Then $\phi(v_2v_6) = 3$, but since one of the edges $v_6v_7$ and $v_5v_6$ must have color 3, we have a contradiction.  
	\end{proof}

	\begin{figure}[h!t]	
		\includegraphics[scale=0.36]{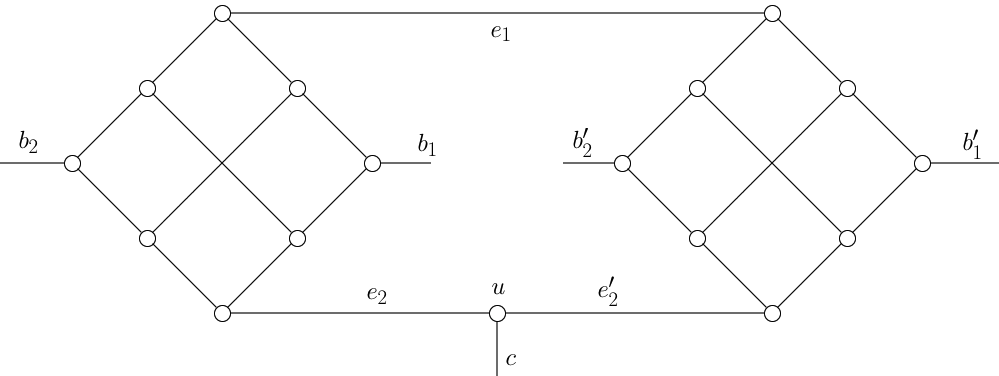}
		\caption{The 5-pole $H_1$.}
		\label{H1-fig}
	\end{figure}
	
	\begin{figure}[h!t]	
		\includegraphics[scale=0.36]{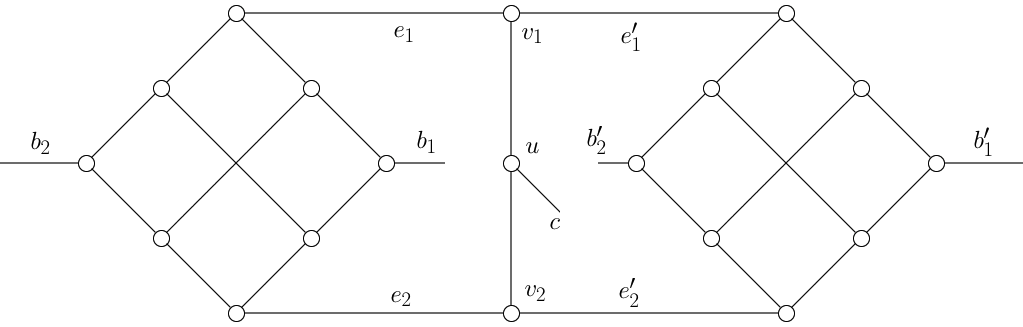}
		\caption{The 5-pole $H_2$.}
		\label{H2-fig}
	\end{figure}
	Now consider the 5-pole $H_1$ obtained the following way: Take two copies $B_1$ and $B_2$ of $B$ and identify the semiedges $b_1$ from $B_1$ with $b_1$ from $B_2$ to get an edge $e_1$. Then add a vertex $u$ which is connected with the semiedges $b_2$ from both $B_1$ and $B_2$ and denote the corresponding new edges by $e_2$ and $e_2'$ (see Figure \ref{H1-fig}). It is straightforward to verify that this 5-pole cannot be 3-edge colorable. 
	\begin{lemma}\label{h1col}
		$H_1$ is not 3-edge colorable. 
	\end{lemma}
	\begin{proof}
		Assume the opposite. Then by Lemma \ref{samecol} $e_1$ must have the same color as $e_2$, but $e_1$ must also have the same color as $e_2'$ which is impossible since $e_2$ and $e_2'$ are incident edges. 
	\end{proof}
	
	The 5-pole $H_2$ is formed in a similar way by joining two copies of $B$ with three vertices $v_1,v_2$ and $u$ as shown in Figure \ref{H2-fig}. 
	\begin{lemma}\label{h1col}
		$H_2$ is not 3-edge colorable. 
	\end{lemma}
	\begin{proof}
		Assume the the graph has a 3-edge coloring. Then by Lemma \ref{samecol} $e_1$ must have the same color as $e_2$ and $e_1'$ must have the same color as $e_2'$. W.l.o.g. we may assume that $e_1$ and $e_2$ has color 1 and $e_1'$ and $e_2'$ has color 2. But then both $v_1u$ and $v_2u$ must have color 3, which is impossible. 
	\end{proof}
	We can now use that fact that any cubic graph which contains either $H_1$ or $H_2$ as a subgraph cannot be colorable to create two infinite families of snarks. 
	
	\subsubsection*{Construction 1.}
	Let $G$ be any 2-edge connected cubic graph and let $D$ be a 2-regular subgraph of $G$. Let $C=(v_1,v_2,\dots,v_k)$ be a cycle in D. Now, form a graph by removing all edges $v_iv_{i+1}$ (indices are taken modulo $k$). Then add $k$ copies $B_1,\dots,B_k$ of the 4-pole $B$ and denote the semiedge edges in $B_i$ by $a_l^i, b_l^i$ for $l\in\{1,2\}$. Now, for all $i\in\{1,\dots, k\}$, connect $v_i$ with the semiedge $b_1^i$, the semiedges $a_2^i$ with $b_2^{i+1}$ and $a_1^i$ with $v_{i+1}$. Repeat this process for every cycle in $D$ and denote the resulting proper graph by $G'$. 
	
	We call $G'$ a \emph{semi blowup} of $(G,D)$ and denote $G'$ by $\textrm{SemiBlowup}(G,D)$. 
	\begin{figure}[h!t]	
		\includegraphics[scale=0.46]{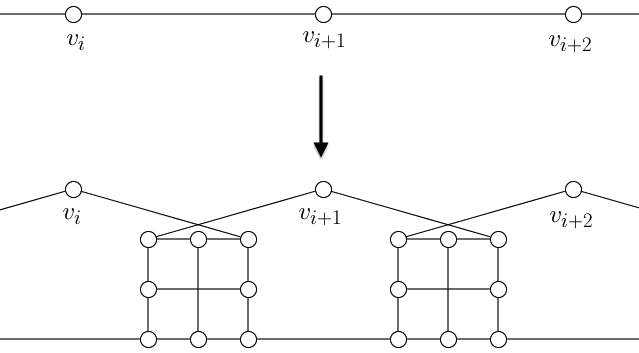}
		\caption{An illustration of Construction 1.}
		\label{blowup-fig}
	\end{figure}
	
	\subsubsection*{Construction 2.}
	The second construction is similar to the previous one. Let $G$ be any 2-edge connected cubic graph and let $D$ be a 2-regular subgraph of $G$. Furthermore let $C=(v_1,v_2,\dots,v_k)$ be a cycle in $D$ and remove all edges $v_iv_{i+1}, i = 1,\dots,k$ and add $k$ copies $B_1,\dots, B_k$ of $B$. Now add two additional vertices $u_i$ and $w_i$. Then add the edges $\{v_iu_i, v_iw_i\}$ and form edges from the semiedges by connecting $u_i$ with $b_1^i$, $w_i$ with $b_2^i$, $a_1^i$ with $u_{i+1}$ and $a_2^i$ with $w_{i+1}$. We continue this process for every cycle in $D$ and call the resulting graph $G''$. We say that $G''$ a \emph{blowup} of $D$ and $G''$ is denoted by $\textrm{Blowup}(G,D)$. 
	
	\begin{figure}[h!t]	
		\includegraphics[scale=0.46]{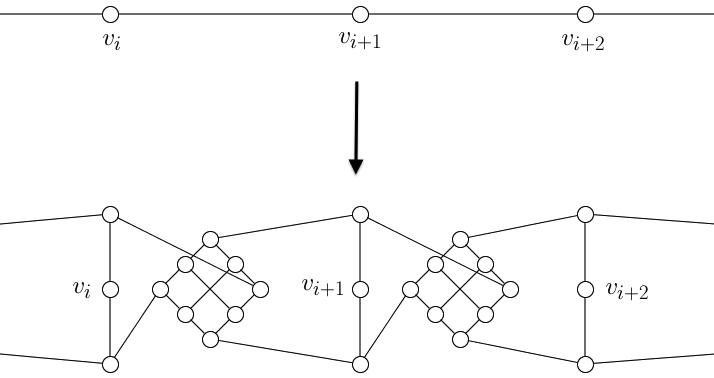}
		\caption{An illustration of construction 2.}
		\label{blowup-fig}
	\end{figure}
	
	\begin{theorem} \label{thm:blowup}
		Let $G$ be a 3-edge connected cubic graph with a 2-regular subgraph $D$. Then $Blowup(G,D)$ and $SemiBlowup(G,D)$ are not 3-edge colorable. Furthermore if $D_1,\dots, D_l$ are the disjoint cycles of $D$, we have that  $o(Blowup(G,D))\geq r_3(Blowup(G,D))\geq \sum_{i=1}^l \lceil\frac{|V(D_i)|}{2}\rceil$ and the same inequality holds for $SemiBlowup(G,D).$
	\end{theorem}
	\begin{proof}
		Let $G_1:=\textrm{SemiBlowup}(G,D)$ and $G_2:=\textrm{Blowup}(G,D)$. Every pair of adjacent edges on a cycle in $D$ gives a subgraph isomorphic to $H_1$ in $G_1$ and a subgraph isomorphic to $H_2$ in SemiBlowup$(G,D)$. Since a cubic graph cannot be colorable if either $H_1$ or $H_2$ are subgraphs, we have that $r_3(G_1)$ and $r_3(G_2)$ is at least the number of edges that needs to be removed in order to make $G_1$ $H_1$-free or $G_2$ $H_2$-free. It is easy to see that this is at least $\sum_{i=1}^l\beta(D_i)$ where $\beta$ is the vertex cover number of a graph (the minimum number of vertices needed to cover all the edges in a graph).
		Since an optimal vertex cover of an even cycle is to choose every second vertex we have that $\beta(C_{2k})=k$ and for odd cycles at least two adjacent vertices must be in the vertex cover, so $\beta(C_{2k+1})= k+1$. Hence \[o(G_j) \geq r_3(G_j)\geq \sum_{i=1}^l\beta(D_i)=\sum_{i=1}^l \Big\lceil\frac{|V(D_i)|}{2}\Big\rceil\]
		for $j=1,2$. 
	\end{proof}
	If $G$ has no cyclic $k$-edge-cuts where $k\leq 3$, then both blowup and semiblowup will produce snarks. This also gives a simple construction of cyclically 4-edge connected snarks with rather high oddness compared to the number of vertices. 
	
	By considering the semiblowup of a hamiltonian cycle in a hamiltonian cubic graph, we get the following: 
	\begin{corollary} \label{cor:bigo1}
		For every $k\in\mathbb N, k>2$ there exists a snark $G_k$ with $|V(G_k)|=18k$ and $o(G_k)\geq k$. 
	\end{corollary}
	This can be improved slightly by instead use many short cycles of odd lengths.
	\begin{corollary}\label{cor:bigo}
		For every $k\in\mathbb N$ there exists a snark $G_k$ with $|V(G_k)|=90k$ and $o(G_k)\geq 6k$. 
	\end{corollary}
	\begin{proof}
		Given $k\in \mathbb N$, let $F_k$ be the 2-factor formed from $2k$ copies of $C_5$. Then add $10k$ edges between the cycles in $F$ in order to obtain a cubic graph $G_k'$. Now let $G_k:=\textrm{SemiBlowup}(G',F)$. Since $\beta(C_5) = 3$ we get $o(G_k)\geq 2k\beta(C_5)=6k$.  
	\end{proof}
	In Corollary \ref{cor:bigo} the oddness of $G_k$ grows linearly in the number of vertices. Obviously it is impossible to construct a family of snarks where the oddness grows superlinearly. Let \[q_k:=\max\Big\{\frac{k}{|V(G)|}:G\in\mathcal G\textrm{ and }o(G)\geq k\Big\}\] 
	where $\mathcal G$ is the family of snarks. 
	The Petersen graph is the smallest snark so $q_2=\frac{1}{5}$. It is known that the smallest snark with oddness 4 has at least 38 vertices \cite{hagglund10} and from Theorem \ref{thm:blowup} we can construct a snark with oddness 4 on 46 vertices (use the semiblowup construction on the cubic graph on 6 vertices which consists of a 5-cycle, a cord and a $K_{1,3}$). Hence, $\frac{1}{12}\leq q_4\leq \frac{2}{18}$. From Corollary \ref{cor:bigo} we get $q_{6k}\geq\frac{1}{15}$ and from Corollary \ref{cor:bigo1} we get $q_k\geq\frac{1}{18}$. We pose the following two problems: 
	\begin{problem}
		Let $k$ be a given even number. Determine $q_k$ when $k>2$.  
	\end{problem}
	\begin{problem}
		What is the largest value $c$ such that $q_k\geq c$ for all even $k$.
	\end{problem}

\section{Perfect matching covers} 
	Let $C$ be a cycle of length 3 in $K_4$ and consider the graph Blowup$(K_4,C)$. 
	By using a computer, it is easy to see that this graph has perfect matching index 5. Note that this graph was also observed in \cite{hagglund10}, as an example of a snark of minimum order for which the removal of any vertex yields a graph homeomorphic to a non 3-edge colorable cubic graph. 
	\begin{observation}\label{obs:tau5}
		$\tau(Blowup(K_4,C))=5$. 
	\end{observation} 
	\begin{figure}[h!t]	
		\includegraphics[scale=0.46]{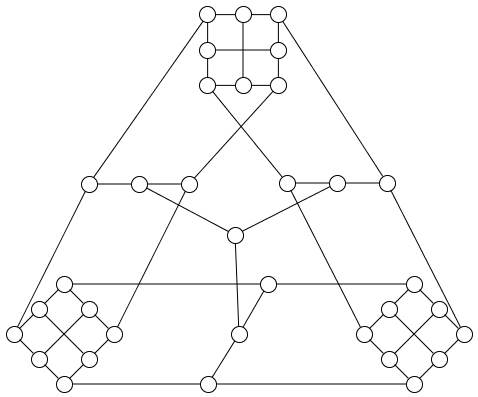}
		\caption{An illustration of Blowup$(K_4,C)$.}
		\label{blowup-fig}
	\end{figure}
	
	It is possible to construct other snarks with this property. Another example is the blowup of a 4-cycle in the prism depicted in Figure \ref{buprism}. However we do not have any good characterization of the cubic graphs and 2-regular subgraphs for which the blowup-construction yields snarks with perfect matching index 5. 
	
	\begin{figure}[h!t]	
		\includegraphics[scale=0.46]{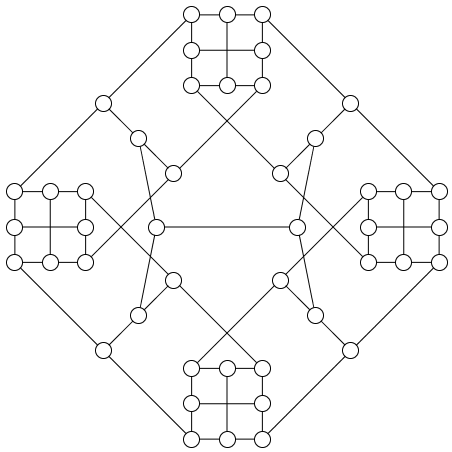}
		\caption{An illustration of Blowup$(Prism,C_4)$.}
		\label{buprism}
	\end{figure}
	
	\begin{problem}
		Is it possible to give a simple characterization of cubic graphs $G$ with $\tau(G)=5$?
	\end{problem}
	Graphs constructed with the blowup-construction always contains cyclic 4-edge cuts, and other than the Petersen graph, we do not know of any cyclically 5-edge connected snarks $G$ with $\tau(G)=5$. 
	\begin{problem}
		Are there any cyclically 5-edge connected snarks $G$, different from the Petersen graph, such that $\tau(G)=5$?
	\end{problem}
	
\section{Cycle double covers} 
	We say that a \emph{cycle double cover} (abbreviated CDC) of a graph is multiset of cycles such that every edge is covered by exactly two cycles. A \emph{k-CDC} is a CDC where we can color assign $k$ colors to the cycles in such a way that if two cycles share an edge, they will always receive different colors. A famous conjecture by Seymour \cite{MR538060} and Szekeres \cite{MR0325438} asserts that every bridgeless cubic graph has a CDC, and a conjecture by Celmis \cite{MR2634187} strengthen this further by asserting that every such graph in fact has a 5-CDC. 
	Celmins also observed a connection between the perfect matching index and the existence of a 5-CDC where one color class induces a 2-factor. 
	\begin{theorem}[Celmins \cite{MR2634187}]\label{thm:celmins}
		Let $G$ be a cubic graph. Then $\tau(G)\leq 4$ if and only if $G$ has a 5-CDC where one color class induces a 2-factor in $G$. 
	\end{theorem}
	From Theorem \ref{thm:celmins} and Observation \ref{obs:tau5} we see that Blowup$(K_4,C)$, where $C$ is a 3-cycle in $K_4$, also has the following interesting property. 
	\begin{corollary}
		No 2-factor in Blowup$(K_4,C)$ can be part of a 5-CDC. 
	\end{corollary}
	
	         By a computer search we also make the following stronger observation.
	\begin{observation}
		No 2-factor in Blowup$(K_4,C)$ can be part of a CDC. 
	\end{observation}
	
	\subsection{Circumference}
	It is known that cubic graphs with oddness 4 or less have cycle double covers \cite{MR2117938,MR1815603}. It is also known that the cycle double cover conjecture holds for graphs with sufficiently long cycles. Given a graph $G$, let the \emph{circumference} of $G$ be the length of the longest cycle. We denote this by $\textrm{circ}(G)$. In \cite{hagglund10:2} it is shown that if a cubic graph $G$ has $\textrm{circ}(G)\geq |V(G)|-9$, then $G$ has a cycle double cover. This result was further improved in \cite{hagglund10} to $|V(G)|-10$. We have already seen that the snarks constructed in Section \ref{section:construct} can have arbitrary high oddness. We will now show that the same constructions also produces snarks with low circumference compared to the number of vertices.  
	
	Given a cubic graph $G$, let $\rho(G)$ denote the minimum number of vertices that needs to be removed from $G$ in order to obtain a 3-edge colorable graph. It was shown in \cite{MR1630478} that this number equals the resistance of the graph.
	\begin{theorem}[Steffen \cite{MR1630478}] \label{thm:steffen}
		Let $G$ be a cubic graph. Then $\rho(G)=r_3(G)$. 
	\end{theorem}
	This theorem was later generalized in \cite{Kochol:2011vn} to apply to all graphs with vertices of degee at most 3. A simple consequence of this theorem is a bound on the circumference. 
	\begin{lemma}\label{lemma:circ}
		Let $G$ be a cubic graph. Then circ$(G)\leq |V(G)|-r_3(G)+1$. 
	\end{lemma}
	\begin{proof}
		Assume that $G$ has a cycle $C$ with $|V(C)|> |V(G)|-r_3(G)+1$. Now consider the graph $G'=G-(V(G)-V(C))$. If $|V(C)|$ is even, then this graph is 3-edge colorable, since we can color the edges of $C$ alternating between colors 1 and 2, and the cords of $C$ with color 3. However this gives $\rho(G)<r_3(G)-1$ contradicting Theorem \ref{thm:steffen}. If, on the other hand, $|V(C)|$ is odd, then obviously $G'-v$ is 3-edge colorable for any vertex $v$ in $C$. We now have $\rho(G)<r_3(G)$ which once again contradicts Theorem \ref{thm:steffen}. 
	\end{proof}
	From Lemma \ref{lemma:circ} and Theorem \ref{thm:blowup} we can see that the graphs obtained from the blowup and semiblowup constructions cannot have too long cycles.
	\begin{corollary}
		Given a cubic graph $G$ and a 2-regular subgraph $D$ with disjoint cycles $C_1,C_2,\dots,C_k$, then $circ(Blowup(G,D))\leq |V(G)|- \sum_{i=1}^k \lceil\frac{|V(C_i)|}{2}\rceil-1$ and the same holds for $SemiBlowup(G,D)$. 
	\end{corollary}
	
\section{Concluding remarks}
	There are other constructions of snarks that are far from 3-edge colorable that are similar to the blowup and semiblowup constructions using other $k$-poles than $B$ (any $k$-pole which is either non-3-edge colorable or where the colors are forced as above works). See e.g. \cite{MR1385382}, \cite{Steffen:2004uq} and \cite{Nedela:1996kx} for more comprehensive studies of this. 
	
	It is somewhat unsatisfactory to require the assistance of a computer to show that the blowup of a 3-cycle in $K_4$ has $\tau=5$. Given a good characterization of this property, it might be possible to construct an infinite family of such snarks. 
	

\bibliographystyle{hplain}
\bibliography{snark_construction}

\begin{thebibliography}{10}

\bibitem{hagglund10}
Gunnar Brinkmann, Jan Goedgebeur, Jonas H{{\"a}}gglund, and Klas
  Markstr{{\"o}}m.
\newblock Generation and properties of snarks.
\newblock Submitted.

\bibitem{MR2634187}
Uldis~Alfred Celmins.
\newblock {\em On cubic graphs that do not have an edge 3-colouring}.
\newblock ProQuest LLC, Ann Arbor, MI, 1985.
\newblock Thesis (Ph.D.)--University of Waterloo (Canada).

\bibitem{Fouquet:2009fk}
Jean-Luc Fouquet and Jean-Marie Vanherpe.
\newblock On the perfect matching index of bridgeless cubic graphs.
\newblock Arxiv, 04 2009, 0904.1296.

\bibitem{MR0294149}
D.~R. Fulkerson.
\newblock Blocking and anti-blocking pairs of polyhedra.
\newblock {\em Math. Programming}, 1:168--194, 1971.

\bibitem{MR2117938}
Roland H{{\"a}}ggkvist and Sean McGuinness.
\newblock Double covers of cubic graphs with oddness 4.
\newblock {\em J. Combin. Theory Ser. B}, 93(2):251--277, 2005.

\bibitem{hagglund10:2}
Jonas H{{\"a}}gglund and Klas Markstr{{\"o}}m.
\newblock On stable cycles and cycle double covers of graphs with large
  circumference.
\newblock {\em Discrete Mathematics}, 2011.

\bibitem{MR1815603}
Andreas Huck.
\newblock On cycle-double covers of graphs of small oddness.
\newblock {\em Discrete Math.}, 229(1-3):125--165, 2001.
\newblock Combinatorics, graph theory, algorithms and applications.

\bibitem{Huck:1995fk}
Andreas Huck and Martin Kochol.
\newblock Five cycle double covers of some cubic graphs.
\newblock {\em J. Combin. Theory Ser. B}, 64(1):119--125, 1995.

\bibitem{MR0382052}
Rufus Isaacs.
\newblock Infinite families of nontrivial trivalent graphs which are not {T}ait
  colorable.
\newblock {\em Amer. Math. Monthly}, 82:221--239, 1975.

\bibitem{MR1385382}
Martin Kochol.
\newblock Snarks without small cycles.
\newblock {\em J. Combin. Theory Ser. B}, 67(1):34--47, 1996.

\bibitem{Kochol:2011vn}
Martin Kochol.
\newblock Three measures of edge-uncolorability.
\newblock {\em Discrete Math.}, 311(1):106--108, 2011.

\bibitem{pre05957651}
G.~Mazzuoccolo.
\newblock {The equivalence of two conjectures of Berge and Fulkerson.}
\newblock {\em J. Graph Theory}, 68(2):125--128, 2011.

\bibitem{Mazzuoccolo:2011fk}
Giuseppe Mazzuoccolo.
\newblock Covering a 3-graph with perfect matchings.
\newblock Arxiv, 11 2011, 1111.1871.

\bibitem{Nedela:1996kx}
Roman Nedela and Martin {\v{S}}koviera.
\newblock Decompositions and reductions of snarks.
\newblock {\em J. Graph Theory}, 22(3):253--279, 1996.

\bibitem{MR538060}
P.~D. Seymour.
\newblock Sums of circuits.
\newblock In {\em Graph theory and related topics ({P}roc. {C}onf., {U}niv.
  {W}aterloo, {W}aterloo, {O}nt., 1977)}, pages 341--355. Academic Press, New
  York, 1979.

\bibitem{MR1630478}
Eckhard Steffen.
\newblock Classification and characterizations of snarks.
\newblock {\em Discrete Math.}, 188(1-3):183--203, 1998.

\bibitem{Steffen:2004uq}
Eckhard Steffen.
\newblock Measurements of edge-uncolorability.
\newblock {\em Discrete Math.}, 280(1-3):191--214, 2004.

\bibitem{MR0325438}
G.~Szekeres.
\newblock Polyhedral decompositions of cubic graphs.
\newblock {\em Bull. Austral. Math. Soc.}, 8:367--387, 1973.

\bibitem{Tutte:1954bh}
W.~T. Tutte.
\newblock A contribution to the theory of chromatic polynomials.
\newblock {\em Canadian J. Math.}, 6:80--91, 1954.

\end{thebibliography}

\end{document}